\documentclass[11pt]{amsart}

\usepackage{mathrsfs,graphicx,latexsym,tikz,color,euscript}
\usepackage{amsfonts,amssymb,amsmath,amscd,amsthm}
\usepackage{epstopdf}
\usepackage{hyperref}
\usepackage{upgreek}
\usepackage{comment}
\usepackage{pdfsync}

\newcommand{\rr}{\mathbb R}
  
\newcommand{\bs}{\mathbb S}

\newcommand{\bb}{\mathbb B}
\newcommand{\A}{\mathcal A}
 
\newcommand{\E}{\mathbb E}

\newcommand{\al}{\alpha}
\newcommand{\gm}{\gamma}

\newcommand{\dl}{\delta}

\newcommand{\ey}{\frac{1}{2}}

\newcommand{\gmd}{\dot{\gm}}

\newcommand{\ac}{\mathcal{C}}

\newcommand{\vv}{\underline{v}}
\newcommand{\dv}{\dl \underline{v}}

\newcommand{\ttl}{\tilde{t}}

\newcommand{\lnm}{\left \|}
\newcommand{\rnm}{\right \|}
\newcommand{\cal}{C_{\al, v}}
\newcommand{\fal}{f_{\al}}

\theoremstyle{plain}
\newtheorem{thm}{Theorem}[section]
\newtheorem{prop}{Proposition}[section]
\newtheorem{lem}{Lemma}[section]

\theoremstyle{definition}
\newtheorem{dfn}{Definition}[section]

\theoremstyle{remark}
\newtheorem{rem}{Remark}[section]

\begin{document}

\title[hyperbolic motions]{Hyperbolic motions in the $N$-body problem with homogeneous potentials}
 

\author{Guowei Yu}
\address{Chern Institute of Mathematics and LPMC, Nankai University, Tianjin, China}
\email{yugw@nankai.edu.cn}

\thanks{This work is supported by the National Key R\&D Program of China (2020YFA0713303), NSFC (No. 12171253), the Fundamental Research Funds for the Central Universities and Nankai Zhide Fundation.}

\begin{abstract}
In the $N$-body problem, a motion is called hyperbolic, when the mutual distances between the bodies go to infinity with non-zero limiting velocities as time goes to infinity. For Newtonian potential, in \cite{MV20} Maderna and Venturelli proved that starting from any initial position there is a hyperbolic motion with any prescribed limiting velocities at infinity. 

Recently based on a different approach, Liu, Yan and Zhou \cite{LYZ21} generalized this result to a larger class of $N$-body problem. As the proof in \cite{LYZ21} is quite long and technical, we give a simplified proof for homogeneous potentials following the approach given in the latter paper.  
\end{abstract}
{}
\maketitle

\section{Introduction}

Consider the $N$-body problem in $\rr^d$ ($d \ge 2$) with $\al$-homogeneous potential as below
\begin{equation}
\label{eq: U} U(q) = \sum_{1 \le i < j \le N} \frac{m_i m_j}{|q_i -q_j|^{\al}}, \;\; \al \in (0, 2).
\end{equation}
With $m_i$ and $q_i$ representing the mass and position of the $i$-th body, motion of the bodies can be described by the following equation
\begin{equation}
\label{eq;Nbody} m_i \ddot{q}_i = - \al \sum_{j=1, j \ne i}^N \frac{m_i m_j(q_i -q_j)}{|q_i -q_j|^{\al +2}}, \;\; i =1, \dots, N. 
\end{equation}
When $\al =1$, it is the Newtonian $N$-body problem. 

Define the mass weighted inner product and norm on the configuration space $\E=\rr^{dN}$ as
$$ \ll x, y \gg = \sum_{i =1}^N m_i\langle x_i, y_i \rangle, \; \|x \| = \sqrt{\ll x, x \gg}. $$
$\bs = \{ x \in \E: \; \|x \|=1\}$ represents the set of normalized configurations.

Since the potential function $U$ is not well-define on the collision set 
\begin{equation}
\label{eq; Delta} \Delta = \{q \in \E: \; q_i = q_j \text{ for some } 1 \le i \ne j \le N\}. 
\end{equation} 
Given an initial condition, the corresponding solution of \eqref{eq;Nbody} may not exist after a finite time. However it is believed (see \cite{Saari71} and \cite{Knauf19}), these initial conditions should form a set with measure zero in the phase space. Now a fundamental question is what are the possible final motions of the bodies as time goes to infinity. This was first studied by Chazy in \cite{Chazy22}. Among all possible final motions, one is called \emph{hyperbolic}, which can be defined as below according to Chazy. 
\begin{dfn}
A solution $q(t)$, $t \in [t_0, \infty)$, of \eqref{eq;Nbody} is a hyperbolic motion, if there is a limiting velocity $\xi \in \hat{\E} = \E \setminus \Delta$ as $t \to \infty$, i.e.,  
\begin{equation}
\label{eq; lim-xi}  \lim_{t \to \infty} \dot{q}(t) = \xi \text{ or } q(t) = \xi t + o(t), \; \text{ as } t \to \infty. 
\end{equation} 
\end{dfn}
Recall that energy $E(q(t))$ is conserved along a solution   
\begin{equation}
\label{eq; energy} E(q(t)) = \ey \|\dot{q}(t)\|^2 - U(q(t)). 
\end{equation}
If $h$ is the energy of a hyperbolic motion $q(t)$, then it must be positive, as
$$ h = \lim_{t \to \infty} E(q(t)) = \lim_{t \to \infty} \left(\ey \|\dot{q}(t)\|^2 - U(q(t) \right) = \ey \|\xi\|^2 >0. $$  
Then we can rewrite the asymptotic expression in \eqref{eq; lim-xi} as below, if  $v = \xi /\sqrt{2h} \in \hat{\bs}= \bs \setminus \Delta$,
\begin{equation}
\label{eq;lim-v} q(t) = \sqrt{2h} v t + o(t), \; \text{ as } t \to \infty. 
\end{equation}
This means $v$ is also the limiting shape of the normalized configuration of $q(t)$, as  
$$ \lim_{t \to \infty} q(t)/\|q(t)\| = v. $$

For the Newtonian $N$-body problem, in \cite{MV20} Maderna and Venturilli showed that starting from any initial configuration, there are hyperbolic motions with any asymptotic limiting velocities from $\hat{\E}$ at infinity, as stated in Theorem \ref{thm;hyper}. The motions were found as free-time minimizers based on the weak KAM theory. 
A crucial part of the proof is to establish the asymptotic expression \eqref{eq;lim-v}. In \cite{MV20}, this was done indirectly using the continuous property of the limiting shape of hyperbolic motions at infinity, which were given by Chazy.

Recently in \cite{LYZ21}, Liu, Yan and Zhou gave a different proof of this result. In particular, they found a way to show the asymptotic expression directly. Moreover their approach works for a much larger class of $N$-body problem, as long as $U(q)$ satisfies certain conditions (including the $\al$-homogeneous potential), for the details see \cite[Theorem 1.4 ]{LYZ21}. To obtained such a general result, the details in \cite{LYZ21} are quite long and technical. 

Since the $\al$-homogeneous potential is most interesting to us and using the Lagrange-Jacobi identity (Lemma \ref{lem;Lag-Jacobi}), one can significantly simplify the proof in Section 5 of \cite{LYZ21} (see Lemma \ref{lem;dist-estimate}). We feel it is worth to write down this short note. However it must be emphasized that the main idea of our proof is from \cite{LYZ21}.

\begin{thm} \label{thm;hyper}
Given arbitrarily a positive energy $h$ and asymptotic configuration $v \in \hat{\bs}$. For any $x \in \E$, there is a hyperbolic motion $\gm: [t_0, \infty) \to \E$, which is collision-free of \eqref{eq;Nbody} for all $t \ne 0$ and satisfies
$$ \gm(t_0) =x, \;\;  \gm(t) = \sqrt{2h}t v + o(t), \; \text{ as } t \to \infty. $$
\end{thm}

In a recently preprint by Polimeni and Terracini \cite{PT23}, another approach to prove the above theorem when $\al=1$ was given. Their idea is to look for minimizer of a normalized Lagrangian action with the desired asymptotic behavior. The advantage of such an approach is it can also be used to prove a similar result for hyperbolic-parabolic motion.

\section{Preliminary}
For any $x, y \in \E$ and $T>0$, let $\ac(x, y;T)$ denote the set of absolute continuous paths defined on $[0, T]$, which go from $x$ to $y$, and $\ac(x,y) = \cup_{T>0} \ac(x, y;T)$. For any $\gm \in \ac(x, y; T)$, we define its length and $h$-modified Lagrangian action value as 
\begin{equation}
\label{eq;length-gm} \ell(\gm|_{[0, T]}) = \int_{0}^{T} \|\dot{\gm}\| \,dt \; \text{ and } \;  \A_h(\gm; 0, T) = \int_{0}^T \left( \ey \|\dot{\gm}\|^2 + U(\gm) +h \right) \,dt.
\end{equation}

\begin{lem}
\label{lem;length-phi} For any $\gm \in \ac(x, y; T)$, $\|y -x\| \le \ell(\gm|_{[0, T]}) \le \frac{1}{\sqrt{2h}}\A_h(\gm; 0, T)$. 
\end{lem}

\begin{proof}
While the first inequality is obvious, the second one follows from
$$ \A_h(\gm; 0, T) \ge \int_0^T \left( \ey \|\gmd\|^2 + h  \right) \,dt \ge \sqrt{2h} \int_0^T \|\gmd \| \,dt. $$
\end{proof}

\begin{dfn}
We say $\gm \in \ac(x, y; T)$ is a minimizer, if $\A_h(\gm; 0, T)= \phi_h(x, y; T)$, and an $h$-free-time minimizer, if $\A_h(\gm; 0, T) = \phi_h(x, y)$, where
$$ \phi_h(x, y;T) = \inf\{ \A_h(\gm; 0, T): \; \gm \in \ac(x, y;T); $$
$$ \phi_h(x, y) = \inf \{\phi_h(x, y; T): \; T>0 \}. $$
\end{dfn}

\begin{prop}
\label{prop;Marchal} If $\gm \in \ac(x, y; T)$ is a minimizer, $\gm|_{(0, T)}$ is a collision-free solution of \eqref{eq;Nbody}. 
\end{prop}

\begin{rem}
This is the well-known \emph{Machal's Lemma}. A proof can be found in \cite{C02} for $\al=1$ and in \cite{FT04} for $\al \in (0, 2)$. Notice that $\gm(0)$ or $\gm(T)$  may contain collision,  if we choose $x$ or $y \in \Delta$.  
\end{rem}

\begin{lem}
\label{lem;h-free-time} When $h>0$ and $x \ne y \in \E$, the following properties hold. 
\begin{enumerate}
\item[(a).] There exists a $\gm \in \ac(x, y;T)$ satisfying $\A_h(\gm; 0, T) = \phi_h(x, y)$. 
\item[(b).] $\gm|_{(0, T)}$ is collision-free solution of \eqref{eq;Nbody} with energy $E(\gm(t)) \equiv h$.
\item[(c).] For any $z \in \E$, $\phi_h(x, y) \le \phi_h(x, z) + \phi_h(z, y)$. 
\end{enumerate}
\end{lem}
\begin{proof}
(a). By a standard argument in the direct method of calculus of variation, there exists a $\gm_\tau \in \ac(x, y; \tau)$ for some $\tau>0$, such that $\A(\gm_{\tau}; 0, \tau) = \phi_h(x, y; \tau)$. The Cauchy-Schwartz inequality then implies
$$ \ell(\gm_{\tau}|_{[0, \tau]}) = \int_{0}^\tau \|\dot{\gm}_\tau\| \,dt \le \sqrt{\tau} \left( \int_0^\tau \|\dot{\gm}_\tau\|^2 \,dt \right)^{\ey}. $$
As $x \ne y$, if $\tau \to 0$, 
$$ \phi_h(x, y; \tau) \ge \ey \int_0^\tau \|\dot{\gm}_\tau\|^2 \,dt \ge \frac{1}{2\tau} \left( \ell(\gm_{\tau}|_{[0, \tau]}) \right)^2 \ge \frac{\|x-y\|^2}{2\tau} \to \infty.
$$
Meanwhile $\phi_h(x, y; \tau) \ge h \tau \to \infty$, when $\tau \to \infty$.

As a result, there must exist a $T>0$ and $\gm \in \ac(x, y; T)$ satisfying
$$ \A_h(\gm; 0, T) = \phi_h(x,y;T) = \inf_{\tau >0} \phi_h(x, y; \tau) = \phi_h(x, y). $$

(b). By Proposition \ref{prop;Marchal}, $\gm|_{(0, T)}$ is a collision-free solution. Since $\gm$ is an $h$-free-time minimizer, it is well-known its energy must be $h$. For a detailed proof see \cite[3.3]{CI00}.

(c). The result is trivial, when $z =x$, $\phi_h(x, x)=0$. The same argument holds for $z =y$. 

When $z \ne x$ and $z \ne y$, there exist $\gm_x \in \ac(x, z)$ and $\gm_y \in \ac(z, y)$ satisfying 
$$ \int L(\gm_x, \dot{\gm}_x) + h\,dt = \phi_h(x, z), \;\; \int L(\gm_y, \dot{\gm}_y) + h \,dt = \phi_h(z, y).
$$
As the concatenation of $\gm_x$ and $\gm_y$ belongs to $\ac(x, y)$, we have 
$$ \phi_h(x, y) \le \phi_h(x, z) + \phi_h(z, y). $$

\end{proof}

\begin{lem}
\label{lem;phi-h-up-bound} There are two positive constants $\beta_1$ and $\beta_2$, such that for any $x, y \in \E$ and $h>0$, 
$$ \phi_h(x, y) \le \left( \beta_1 \|x-y\|^2 + \beta_2 \|x-y\|^{2-\al} \right)^{\ey}. $$
\end{lem}

\begin{proof}
When $x = y$, $\phi_h(x, y)=0$ and the result holds. Let's assume $x \ne y$ from now on. 

For any $r >\|x-y\|>0$, by \cite[Theorem 1]{Md2012}, there two positive constants $\eta_1$ and $\eta_2$ independent of $x$ and $y$, such that for any $T>0$, 
$$ \phi_h(x, y; T) \le \eta_1 \frac{r^2}{T} + \eta_2 \frac{T}{r^{\al}} + h T. $$
The right hand side of the above inequality as a function of $T$ has a global minimum in $(0, \infty)$ at $T = (\eta_1 r^2/(h + \eta_2/r^{\al}))^\ey$, so 
$$ \phi_h(x, y) \le \left( 4 h \eta_1 r^2 + 4 \eta_1 \eta_2 r^{2-\al} \right)^{\ey}. $$
Letting $r$ go to $\|x-y\|$, the desired result follows from continuity with $\beta_1 = 2 h \eta_1$ and $\beta_2 = 4 \eta_1 \eta_2$. 
\end{proof}

This lemma further implies continuity of $\phi_h$. 

\begin{lem}
\label{lem;phi-h-cont} $\phi_h(x, y)$ is continuous with respect to both $x$ and $y$. 
\end{lem}

\begin{proof}
Choose an arbitrary $y'$ close to $y$, by Lemma \ref{lem;phi-h-up-bound}, there is a $C>0$, such that 
$$ \phi_h(x, y) - \phi_h(x, y') \le \phi_h(y', y) \le C \|y-y'\|^{\frac{2-\al}{2}}; $$
$$ \phi_h(x, y') - \phi_h(x, y) \le \phi_h(y, y') \le C\|y-y'\|^{\frac{2-\al}{2}}. $$
This shows $\phi_h(x, y)$ is continuous with respect to $y$. The proof for $x$ is similar. 
\end{proof}

\begin{lem}
\label{lem;Lag-Jacobi} If $\gm(t)$ is collision-free solution of \eqref{eq;Nbody} with energy $h>0$, then $ d^2\|\gm(t)\|/dt^2 >0$. 
\end{lem}

\begin{proof}
Since $\gm(t)$ is collision-free, $\|\gm(t)\|$ is always positive. Let $I(t) = \|\gm(t)\|^2$. It is enough to prove $\ddot{I}(t) >0$, which follows from 
$$ 
\begin{aligned}
\ey \ddot{I} & = \ll \dot{\gm}, \dot{\gm} \gg^2 + \ll \gm, \ddot{\gm} \gg = \|\dot{\gm}\|^2 -\al U(\gm) \\
& =\frac{2-\al}{2} \|\dot{\gm}\|^2 + \al \left( \ey \|\dot{\gm}\|^2 -U(\gm) \right) = \frac{2-\al}{2} \|\dot{\gm}\|^2 + \al h>0.  
\end{aligned}
$$
\end{proof}

For any $v \in \hat{\bs}$, define $\vv:=\min \{ |v_{ij}| =|v_i -v_j|: \; 1 \le i \ne j \le N \}$ and 
$$ \bb_{\dl \vv}(v) :=\{ u \in \bs: \|u-v\| \le \dl \vv\}, $$
where $\dl :=  \frac{\sqrt{m_0}}{4}$ and $m_0 := \min\{m_i: \; i =1, \dots, N\}. $  
\begin{lem}
\label{lem;v-u}  If $u \in \bb_{\dl \vv}(v)$, then $|u_{ij}| \ge \ey |v_{ij}|$, $\forall 1 \le i \ne j \le N,$ and $U(u) \le 2^{\al} U(v).$
\end{lem}
\begin{proof}
Notice that $m_0 |u-v|^2 \le \|u-v\|^2$. Then
$$ \begin{aligned}
|u_{ij}| & = |u_i -v_i +v_i -v_j +v_j -u_j| \ge |v_{ij}| - |u_i -v_i| - |u_j -v_j|\\
& \ge |v_{ij}| - 2|u-v| \ge |v_{ij}| - \frac{2}{\sqrt{m_0}} \|u-v\| \ge |v_{ij}| - \ey \vv \\
& \ge |v_{ij}| - \ey |v_{ij}| = \ey |v_{ij}|. 
\end{aligned}
$$
By the above inequality, 
$$ U(u) = \sum_{1 \le i < j\le N} \frac{m_i m_j}{|u_{ij}|^{\al}} \le \sum_{1 \le i < j \le N} 2^{\al} \frac{m_i m_j}{|v_{ij}|^{\al}} \le 2^{\al} U(v). $$
\end{proof}

\begin{lem} \label{lem;Action-xi}
If $u \in \bb_{\dl \vv}(v)$ and $r_2>r_1>0$, $\frac{1}{\sqrt{2h}}\phi_h(r_1u, r_2u) \le r_2-r_1 + W_{\al, v}(r_1, r_2)$ with
\begin{equation}
\label{eq;W-al-v}  W_{\al, v}(r_1, r_2) = \begin{cases}
\frac{2^{\al}U(v)}{2h(\al-1)  }\left(\frac{1}{r_1^{\al-1}}-\frac{1}{r_2^{\al-1}} \right), & \text{ if } \al >1; \\
\frac{2^{\al}U(v)}{ 2h}(\log r_2- \log r_1), & \text{ if } \al =1; \\
\frac{2^{\al}U(v)}{2h(1-\al) }(r_2^{1-\al}-r_1^{1-\al}), & \text{ if } \al \in (0, 1). 
\end{cases}
\end{equation}
\end{lem}

\begin{proof}
Let $\xi(t) = r_1u + \sqrt{2h}tu$, $t \in [0, \frac{r_2-r_1}{\sqrt{2h}}]$. Then $\rho(t) = \|\xi(t)\| = r_1 + \sqrt{2h}t$ and 
$$ \frac{1}{\sqrt{2h}}\A_h(\xi; 0, \frac{r_2-r_1}{\sqrt{2h}}) \le r_2-r_1 +  \frac{2^{\al}U(v)}{2h} \int_{r_1}^{r_2} \rho^{-\al} \,d\rho = r_2-r_1 + W_{\al, v}(r_1, r_2).
$$ 
\end{proof}

\begin{lem} \label{lem;dist-estimate}
Given an arbitrary $x$ and $v \in \hat{\bs}$. There is an $R> 2\|x\|$ large enough, such that for any $y$ with  $\|y\| \ge R$ and $u = y /\|y\| \in \bb_{\dv}(v)$, if $\gm \in \ac(x, y; T)$ is an $h$-free-time minimizer, there is a constant $C_{\al, v}$ independent of $\gm$, such that
$$ \lnm \frac{\gm(t)}{\|\gm(t)\|} - u \rnm \le  C_{\al, v}f_{\al}(\|\gm(t)\|), \; \forall t \in [T', T],  
$$
where $T'= \inf \{ t \in [0, T]; \; \|\gm(t) \| = \ey \|y\| \}$ and $\fal$ is decreasing function as below
\begin{equation} \label{eq;f-al} 
\fal(r) = 
\begin{cases}
  (2 r)^{-\ey}, \; & \text{ if } \al > 1; \\
  \left(\frac{\log 2r}{2r} \right)^{\ey}, \; & \text{ if } \al=1; \\
  (2r)^{-\frac{\al}{2}} , \; & \text{ if } \al \in (0, 1). \\
\end{cases}
\end{equation}
\end{lem} 
\begin{proof}
Since $\ey\|y\| > \|x\|$, $T'$ is well-defined with $\frac{d \|\gm(t)\|}{dt}|_{t = T'}>0$. By Lemma \ref{lem;Lag-Jacobi}, $\|\gm(t)\|$ is strictly increasing, when $t \in [T', T]$. 

By Lemma \ref{lem;phi-h-up-bound}, there is a constant $C_1$ depending on $x$ and $v$, such that 
$$ \sup\{ \phi_h(x, \|x\|u); \; u \in \bb_{\dl \vv}(v) \} \le \sqrt{2h} C_1.$$
Then Lemma \ref{lem;h-free-time} and \ref{lem;Action-xi} imply
\begin{equation} \label{eq;phi-UppBnd}
\phi_h(x, y) \le \phi_h(x, \|x\|u) + \phi_h(\|x\|u, y) \le \sqrt{2h} \big(C_1 + \|y\| -\|x\| + W_{\al, v}(\|x\|, \|y\|)\big).
\end{equation}
With this, Lemma \ref{lem;length-phi} implies 
\begin{equation}
\label{eq;ell-gm<phi-h}  
\ell(\gm|_{[0, T]}) \le \frac{1}{\sqrt{2h}} \phi_h(x, y) \le \|y\| + W_{\al, v}(\|x\|, \|y\|)+ C_1  -\|x\|. 
\end{equation}

Fix an arbitrary $t \in [T', T)$ (the result is trivial, when $t = T$), set $r= \|\gm(t)\|$ and $z = ru$, then $\Sigma = \{ q \in \E: \; \ll q-z, u \gg =0\}$ is the hyperplane perpendicular to the line $\overrightarrow{0z}$ at $z$. If $\gm(t) =z$, set $t' =t$. If not, $\gm(t)$ and $y$ must be separated by $\Sigma$, as $\|\gm(t)\| < \|y\|$. Then we can always find a $t' \in (t, T)$ with $\gm(t') \in \Sigma$(see Figure \ref{fig:Triangle}).  As a result, 
\begin{equation}
\begin{cases}
\|\gm(t')- z \|^2 = \|\gm(t') \|^2 - \|z \|^2, & \|\gm(t')\| \ge \|z\| \\
\|\gm(t')- z \|^2 = \|y -\gm(t')\|^2 - \|y -z\|^2, &  \|y -\gm(t')\| \ge \|y-z\|
\end{cases}
\end{equation}

\begin{figure} 
\centering
\includegraphics[scale=0.8]{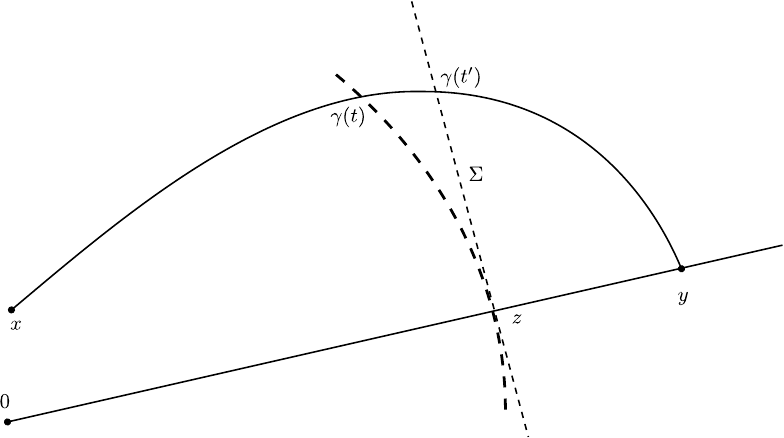}
\caption{}
\label{fig:Triangle}
\end{figure}

Combining these with \eqref{eq;ell-gm<phi-h}, we get
\begin{equation}
\label{eq;gmt'-z-square} \begin{aligned}
& 2\|\gm(t') -z\|^2  = \|\gm(t')\|^2 + \|\gm(t') -y\|^2 - (\|z\|^2 + \|y -z\|^2) \\
& = (\|\gm(t')\| + \|\gm(t') - y\|)^2 - 2 \|\gm(t')\| \cdot \|\gm(t') - y \| -   (\|y-z\| + \|z\|)^2 + 2\|z\| \cdot \|y-z\|  \\
& \le (\|\gm(t')\| + \|\gm(t') - y\|)^2 - \|y\|^2 \\
& \le (\|x\| + \|x -\gm(t')\| + \|\gm(t') -y\|)^2 - \|y\|^2 \\
& \le (\|x\| + \ell(\gm|_{[0, T]}))^2 - \|y\|^2 \\
& \le (\|y\| +W_{\al, v}(\|x\|, \|y\|) + C_1)^2-\|y\|^2
\end{aligned}
\end{equation}
For simplicity let's write $W_{\al, v}(\|x\|, \|y\|)$ as $W_{\al, v}$. Then  
\begin{equation}
\label{eq;gmt'-z} \|\gm(t')-z\| \le \left(  \|y\| W_{\al, v} + \ey W^2_{\al, v} + C_1 (\|y\| + W_{\al, v} + \ey C_1) \right)^{\ey}.
\end{equation}
Meanwhile notice that 
$$ \ell(\gm|_{[0, t]}) \ge \|\gm(t) -x\| \ge r -\|x\|;$$
$$\ell(\gm|_{[t', T]}) \ge \|y -\gm(t')\| \ge \|y-z\| \ge \|y\|-r. $$ 
Then together with \eqref{eq;ell-gm<phi-h}, they imply
\begin{equation*}
\label{eq;l-gm-t-t'} 
\| \gm(t) -\gm(t')\|  \le \ell(\gm|_{[t, t']}) = \ell(\gm|_{[0, T]}) - \ell(\gm|_{[0, t]}) -\ell(\gm|_{[t', T]})  \le W_{\al, v}+ C_1. 
\end{equation*}
Combining this with \eqref{eq;gmt'-z}, we get
\begin{equation}
 \begin{aligned}
 \|\gm(t) - z\| & \le \|\gm(t) -\gm(t')\| + \|\gm(t')- z\| \\
 & \le \left(\|y\| W_{\al, v}+ \ey W^2_{\al, v} +  C_1\Big( \|y\| + W_{\al, v} + \ey C_1 \Big) \right)^{\ey} + W_{\al, v} + C_1. 
\end{aligned}
\end{equation}
Dividing the above inequality by $r=\|\gm(t)\|$, and using the fact that $\|y\| \le 2r$, we get
\begin{equation*}
\label{eq;gmt-z} \begin{aligned}
 \lnm \frac{\gm(t)}{\|\gm(t)\|} - u \rnm & \le \left( \frac{2 W_{\al, v}}{r} + \ey\left( \frac{W_{\al, v}}{r} \right)^2 + \frac{C_1 (2r + W_{\al, v}  + \ey C_1 )}{r^2} \right)^{\ey} + \frac{W_{\al, v}}{r} + \frac{C_1}{r} 
 \end{aligned}
\end{equation*}
From now on we use $W_{\al, v}$ to represent $ W_{\al, v}(\|x\|, 2r)$. As $\|y\| \le 2r$ implies $W_{\al, v}(\|x\|, \|y\|) \le W_{\al, v}(\|x\|, 2r)$, the above inequality still holds. Hence we may rewrite is as,
 $$ \lnm \frac{\gm(t)}{\|\gm(t)\|} - u \rnm \le \left( 4 \frac{W_{\al, v}}{2r} + 2 \left( \frac{W_{\al, v}}{2r} \right)^2 + \frac{C_1 (2r + W_{\al, v}  + \ey C_1 )}{r^2} \right)^{\ey} + 2\frac{W_{\al, v}}{2r} + \frac{C_1}{r} . $$
When $r \to \infty$, by \eqref{eq;W-al-v}, $\frac{W_{\al, v}}{2r} \to \infty$. While 
$$ 2\frac{W_{\al, v}}{2r} + \frac{C_1}{r} \simeq o \left(\sqrt{\frac{W_{\al, v}}{2r}}\right),$$
and
$$ \frac{C_1 (2r + W_{\al, v}  + \ey C_1 )}{r^2} \simeq \begin{cases}
o(\frac{W_{\al, v}}{2r}), & \text{ when } \al \in (0, 1]; \\
O(\frac{W_{\al, v}}{2r}), & \text{ when } \al \in (1, 2). 
\end{cases}
$$ 
As a result, when $R$ is large enough, there is a constant $C_2$ independent of $\gm$ with 
$$  
 \lnm \frac{\gm(t)}{\|\gm(t)\|} - u \rnm \le C_2 \left(\frac{W_{\al, v}(\|x\|, 2 r)}{2r} \right)^{\ey}.
$$
The rest then follows from Lemma \ref{lem;Action-xi} and direct computations.

For the case $x =0$, again by Lemma \ref{lem;phi-h-up-bound}, there is a constant $C_2$, such that 
$$ \sup\{ \phi_h(0, u); \; u \in \bb_{\dl \vv}(v) \} \le \sqrt{2h} C_2. $$
Using  Lemma \ref{lem;h-free-time} and \ref{lem;Action-xi} again, we get 
\begin{equation} \label{eq;phi-h-upp-0}
\phi_h(0, y) \le \phi_h(0,  u) + \phi_h(u, y) \le \sqrt{2h} \big(C_2 + \|y\| -1 + W_{\al, v}(1, \|y\|)\big).
\end{equation}
With this, we get the desired result by repeating the exact same argument as above.
\end{proof}

\begin{lem}
\label{lem;sum-f-al} $\lim_{n_0 \to \infty} \sum_{k=n_0}^{\infty} f_{\al}(2^k) =0$. 
\end{lem}

\begin{proof}
First let's assume $\al \in (0, 1)$, then 
$$ \frac{f_{\al}(2^{k+1})}{f_{\al}(2^k)} = \frac{(2^{k+1})^{\frac{\al}{2}}}{(2^{k+2})^{\frac{\al}{2}}} =  \frac{1}{2^{\frac{\al}{2}}} < 1. $$
By the summation formula of geometric series, 
$$ \sum_{k =n_0}^{\infty} f_{\al}(2^k) = \frac{f_{\al}(2^{n_0})}{1 - 2^{-\frac{\al}{2}}} = \frac{1}{(2^{\frac{\al}{2}}-1)} \frac{1}{2^{\frac{\al}{2}n_0}} \to 0, \; \text{ as } n_0 \to \infty,$$
 
The proof is the same for $\al>1$, as in this case
$$ \frac{f_{\al}(2^{k+1})}{f_{\al}(2^k)} = \frac{(2^{k+1})^{\frac{1}{2}}}{(2^{k+2})^{\frac{1}{2}}} =  \frac{1}{2^{\frac{1}{2}}} < 1,$$
and 
$$ \sum_{k =n_0} \fal(2^k) = \frac{\fal(2^{n_0})}{1-2^{\ey}} = \frac{1}{\sqrt{2}-1} \frac{1}{2^{\frac{n_0}{2}}}  \to 0, \text{ as } n_0 \to \infty.
$$

Now assume $\al =1$. Since  there is an $n_0$ large enough, such that $\log 2^{k+1} \le 2^{\frac{k+1}{2}}$, $\forall k \ge n_0$, 
$$ \sum_{k=n_0}^{\infty} f_{\al}(2^{k}) \le \sum_{k =n_0}^{\infty} 2^{-\frac{k+1}{4}} = \frac{1}{2^{\frac{1}{4}}-1} \frac{1}{2^{\frac{n_0}{4}}} \to 0, \text{ as } n_0 \to \infty.$$
\end{proof}
 
\section{Proof of Theorem \ref{thm;hyper}} \label{sec; thm} 
By Lemma \ref{lem;sum-f-al}, we can find an $n_0$ large enough, such that
\begin{equation}
\label{eq;le-dv} 
 C_{\al,v} \sum_{k=n_0}^{\infty}\fal(2^k) \le  \dl \vv.
\end{equation}
For each $n \ge n_0$, set $y_n = 2^n v$, there is a $\gm_n \in \ac(x, y_n; T_n)$ satisfying $\A_h(\gm_n; 0, T_n) = \phi_h(x, y_n).$ 

Let $R$ be the constant given in Lemma \ref{lem;dist-estimate}, we will further assume $2^{n_0} \ge R$. By the strictly increasing property of $\|\gm_n(t)\|$ obtained in the proof of the same lemma, for each $n \ge n_0$, there is a unique sequence of moments $\{t_{n, k}\}_{k=n_0}^{n}$, such that 
$$ \|\gm_n(t_{n, k})\|= 2^k \; \text{ for each } \; k = n_0, n_0 +1, \dots, n.$$

\begin{lem}
\label{lem;dist-gmn-v} For each $n \ge n_0$ and $t \in [t_{n, n_0}, t_{n, n}]$, let $[\|\gm_n(t)\|]$ be the integer part of $\|\gm_n(t)\|$,  
\begin{equation} \label{eq;dist-gmn-v}
\lnm \frac{\gm_n(t)}{\|\gm_n(t)\|} -v \rnm \le C_{\al, v}\sum_{i =[\log\|\gm_n(t)\|]}^{\infty} f_{\al}(2^i) \le \dl \vv,
\end{equation}
\end{lem}

\begin{proof}
As $\|\gm_n(t)\| \ge 2^{n_0}$, $\forall t \in [t_{n, n_0}, t_{n, n}]$, the second inequality in \eqref{eq;dist-gmn-v} follows directly from \eqref{eq;le-dv}. 

Notice  that $2^{n-1} \le \|\gm_n(t)\| < 2^n$, when $t \in [t_{n, n-1}, t_{n, n})$. By Lemma \ref{lem;dist-estimate},
\begin{equation*}
\lnm \frac{\gm_n(t)}{\|\gm_n(t)\|} -  \frac{\gm_n(t_{n, n})}{\|\gm_n(t_{n, n})\|} \rnm \le C_{\al, v} f_{\al}(\|\gm_n(t)\|) \le C_{\al, v} \fal(2^{n-1}) \le   C_{\al, v} \sum_{i =[\log \|\gm(t)\|]}^{\infty} \fal(2^i). 
\end{equation*}
This proves \eqref{eq;dist-gmn-v}, when $t \in [t_{n, n-1}, t_{n, n})$. 

Now let's assume there is an integer $k \in (n_0, n)$, such that \eqref{eq;dist-gmn-v} holds for all $t \in [t_{n, k+1}, t_{n, n})$. With this we can apply Lemma \ref{lem;dist-estimate} to $\gm|_{[0, t_{n, k+1}]}$, and as a result, for any $t \in [t_{n, k}, t_{n, k+1})$,  
$$ \lnm \frac{\gm_n(t)}{\|\gm_n(t)\|} - \frac{\gm_n(t_{n, k+1})}{\| \gm_n(t_{n, k+1})\|} \rnm \le C_{\al, v} \fal(\|\gm(t)\|) \le C_{\al, v} \fal([\log \|\gm(t)\|]) = C_{\al, v} \fal(2^{k}). $$
As a result, for any $t \in [t_{n, k}, t_{n, k+1})$,  
$$ \begin{aligned}
\lnm \frac{\gm_n(t)}{\|\gm_n(t)\|} - v \rnm & \le \lnm \frac{\gm_n(t)}{\|\gm_n(t)\|} -\frac{\gm_n(t_{n, k+1})}{\|\gm_n(t_{n, k+1})\|} \rnm  + \lnm \frac{\gm_n(t_{n, k+1})}{\|\gm_n(t_{n, k+1})\|} -v \rnm \\
& \le C_{\al, v} \sum_{i = k}^{\infty} f_{\al}(2^i)  \le C_{\al, v} \sum_{i =[\log \|\gm_n(t)\|]}^{\infty} \fal(2^i).
\end{aligned}
$$
The desired result then follows from induction.   
\end{proof}

By the above lemma, $\frac{\gm_n(t)}{\|\gm_n(t)\|} \in \bb_{\dv}(v)$, for any $t \in [t_{n, n_0}, t_{n,n}]$. From the proof of Lemma \ref{lem;dist-estimate}, in particular \eqref{eq;phi-UppBnd} and \eqref{eq;phi-h-upp-0}, we know there is a constant $C_0>0$, such that 
\begin{equation} \label{eq;phi-h-gmt-UpBnd}
\begin{aligned}
\phi_h(x, \gm_n(t)) &  \le \phi_h \left(x, \|x\|\frac{\gm_n(t)}{\|\gm_n(t)\|} \right) + \phi_h \left(\|x\|\frac{\gm_n(t)}{\|\gm_n(t)\|}, \gm_n(t)\right) \\
& \le \sqrt{2h} \big(  C_0  + \|\gm_n(t)\|  + W_{\al, v}(\|x\|, \|\gm_n(t)\|) \big).
\end{aligned}
\end{equation}
Here and in the below, when $x=0$, $W_{\al, v}(\|x\|, \cdot)$ should be seen as $W_{\al, v}(1, \cdot)$. 

As the energy identity can be written as $\|\gmd_n\|^2 = L(\gm_n, \dot{\gm}_n) +h$, we get
\begin{equation}
\label{eq;gmd-sqare-UppBnd} \int_0^t \|\gmd_n(s)\|^2 \,ds = \phi_h(x, \gm_n(t)) \le \sqrt{2h} \big(  C_0  + \|\gm_n(t)\|  + W_{\al, v}(\|x\|, \|\gm_n(t)\|) \big);
\end{equation}  
Meanwhile the energy identity also implies  $\|\gmd_n\| \ge \sqrt{2h}$. Therefore
\begin{equation}
\sqrt{2h}t \le \int_0^t \|\gmd_n \| \,ds \le \sqrt{t} \left( \int_0^t \|\gmd_n\|^2 \,ds \right)^{\ey} \le (\sqrt{2h}t)^{\ey} \big( \|\gm_n(t) \| + W_{\al, v} +C_0  \big)^{\ey}.
\end{equation}
As a result, 
\begin{equation}
 \label{eq;lower-bound} \frac{\sqrt{2h}t}{\|\gm_n(t) \|} \le 1 + \frac{W_{\al, v} + \beta_{x,v}}{\|\gm_n(t)\|} \Rightarrow \frac{\|\gm_n(t)\|}{\sqrt{2h}t} - 1 \ge - \frac{W_{\al, v} + C_0}{\|\gm_n(t)\| + W_{\al, v} + C_0}. 
 \end{equation} 

On the other hand, \eqref{eq;gmd-sqare-UppBnd} also implies
$$ \frac{1}{t} \| \gm_n(t) -x\|^2 \le \frac{1}{t}\left( \int_0^t \|\gmd_n\| \,ds \right)^2 \le \int_0^t \|\gmd_n\|^2 \,ds \le \sqrt{2h}(\|\gm_n(t) \| + W_{\al, v} +C_0).
$$
Multiplying the above inequality by $\|\gm_n(t)\|/(\sqrt{2h}\|\gm_n(t)-x\|^2)$, we get 
$$ \begin{aligned}
\frac{\|\gm_n(t)\|}{\sqrt{2h}t} & \le \frac{ \|\gm_n(t)\| (\|\gm_n(t)\| + W_{\al, v}+ C_0)}{\|\gm_n(t) -x\|^2} \\
& \le \frac{ (\|\gm_n(t)-x\| +\|x\|) (\|\gm_n(t)-x\|+\|x\| + W_{\al, v}+ C_0)}{\|\gm_n(t) -x\|^2}.
\end{aligned}
$$ 
As a result, 
\begin{equation}
\label{eq;upp-bound} \frac{\|\gm_n(t)\|}{\sqrt{2h}t} - 1 \le \frac{W_{\al, v}+ C_0 + 2\|x\|}{\|\gm_n(t)-x\|} + \frac{\|x\|(W_{\al, v} + C_0+ \|x\|)}{\|\gm_n(t) -x\|^2}. 
\end{equation}
By \eqref{eq;lower-bound} and \eqref{eq;upp-bound}, when $n_0$ is large enough, we can find a continuous function $g_{\al}: [2^{n_0}, \infty)$,
\begin{equation}
\label{eq;g-al}  g_{\al}(r) = \begin{cases}
O(\frac{1}{r}), & \text{ if } \al >1; \\
O(\frac{\log r}{r}), & \text{ if } \al =1; \\
O(\frac{1}{r^{\al}}), & \text{ if } \al \in (0, 1),
\end{cases} \text{ as } r \to \infty
\end{equation}
such that, for any $t \in [t_{n, n_0}, t_{n, n}]$,
\begin{equation}
\label{eq:gmn/t}  \left| \frac{\|\gm_n(t)\|}{\sqrt{2h}t} - 1 \right| \le g_{\al}(\|\gm_n(t)\|).
\end{equation}
 
By \eqref{eq;g-al} and \eqref{eq; gmnt/t}, we can assume the following holds 
\begin{equation} \label{eq; gmnt/t}
\ey \le \frac{\|\gm_n(t)\|}{\sqrt{2h}t} \le 2, \;\; \forall t \in [t_{n, n_0}, t_{n, n}] \text{ and } n \ge n_0. 
\end{equation}
The second inequality in \eqref{eq; gmnt/t} then implies 
$$ T_n =t_{n, n} \ge \frac{\|\gm_n(t_{n, n})\|}{2\sqrt{2h}} =\frac{2^{n}}{2\sqrt{2h}} \to \infty, \; \text{ as } n \to \infty. 
$$
On the other hand, the first inequality in  \eqref{eq; gmnt/t} implies 
$$ t_{n, n_0} \le \frac{2}{\sqrt{2h}} \|\gm_n(t_{n, n_0})\| = \frac{2^{n_0+1}}{\sqrt{2h}}, \;\; \forall n \ge n_0. $$
Hence $\ttl_{n_0} =\sup\{t_{n, n_0}:\; n \ge n_0\}$ is a finite number. 

Now fix an arbitary $t \ge \ttl_{n_0}$, by Lemma \ref{lem;dist-gmn-v} $\frac{\gm_n(t)}{\|\gm_n(t)\|} \in \bb_{\dv}(v)$. Then \eqref{eq; gmnt/t} implies  
\begin{equation}
\label{eq;gmnt-norm<t} 2^{n_0} \le \|\gm_n(t)\| \le 2 \sqrt{2h}t, \;\; \forall n \ge n_0.
\end{equation}
Combining the second inequality above with \eqref{eq;phi-h-gmt-UpBnd} gives us 
$$  
\phi_h(x, \gm_n(t)) \le 4ht+  \sqrt{2h} \big(W_{\al, v}(\|x\|, \sqrt{2h}t) + C_0 \big).
$$
We may further assume $n_0$ large enough, such that 
$$ W_{\al, v}(\|x\|, \sqrt{2h}t) + C_0 \le 2\sqrt{2h} t, \; \text{ when } \sqrt{2h}t \ge 2^{n_0}. $$
Hence $\A_h(\gm_n; 0, t) = \phi_h(x, \gm_n(t)) \le 8ht$, $\forall n \ge n_0$. Then for any $0 \le \tau < \tau' \le t$, 
$$ \begin{aligned}
\|\gm_n(\tau') -\gm_n(\tau) \| & \le \int_{\tau}^{\tau'} \|\gmd_n\| \,ds \le (\tau' -\tau)^{\ey} \left( \int_{\tau}^{\tau'} \|\gmd_n\|^2 \,ds \right)^{\ey} \\
&  \le (\tau' -\tau)^{\ey} \big(\A_h(\gm; 0, t)\big)^{\ey}  \le 2\sqrt{2h}(\tau'-\tau)^{\ey} t^{\ey}. 
\end{aligned}$$
This shows $\{\gm_n|_{[0, t]}: n \ge n_0 \}$ is equicontinuous. Since $\gm_n(0) = x$,  $\forall n \ge n_0$, it is also equibounded. By the Ascoli-Arzela theorem (after passing to a proper subsequence) $\{\gm_n|_{[0, t]}; n \ge n_0\}$ will converge uniformly to an absolute continuous path $\xi: [0, t] \to \E$. 

Since the above argument can be repeated for any $\tau_k \ge \ttl_{n_0}$ with $\lim_{k \to \infty} \tau_k= \infty$, by a diagonal argument there exists an absolute continuous path $\gm: [0, \infty) \to \E$, such that $\gm_n$ converges to $\gm$ uniformly on any compact sub-interval of $[0, \infty)$. 

For any $T>0$, by the lower semi-continuity of $\A_h$ and the continuity of $\phi_h$ (Lemma \ref{lem;phi-h-cont}),
$$ \phi_h(x, \gm(T)) \le \A_h(\gm; 0, T) \le \liminf_{n \to \infty} \A_h(\gm_n; 0, T) = \liminf_{n \to \infty} \phi_h(x, \gm_n(T)) = \phi_h(x, \gm(T)). $$
Hence $\gm|_{[0, T]}$ is an $h$-free-time minimizer, and then a collision-free $h$-energy solution of \eqref{eq;Nbody}. 

The above results show that for any $t \ge \ttl_{n_0}$,
$$ \ey \le \frac{\|\gm(t)\|}{\sqrt{2h}t} = \lim_{n \to \infty} \frac{\|\gm_n(t)\|}{\sqrt{2h}t} \le 2;$$
$$ \lnm \frac{\gm(t)}{\|\gm(t)\|} -v\rnm = \lim_{n \to \infty} \lnm \frac{\gm_n(t)}{\|\gm_n(t)\|} -v\rnm \le \dv; 
$$
$$ \left| \frac{\gm(t)}{\sqrt{2h}t}-1 \right| = \lim_{n \to \infty} \left| \frac{\gm_n(t)}{\sqrt{2h}t}-1 \right|  \le  \lim_{n \to \infty}g_{\al} (\|\gm_n(t)\|) \le g_{\al}(\|\gm(t)\|). $$
With these estimates, we get
$$ \begin{aligned}
 \lim_{t \to \infty} \frac{\|\gm(t) - \sqrt{2h}t v\|}{\sqrt{2h}t} & = \lim_{t \to \infty} \lim_{n \to \infty} \frac{\|\gm_n(t) - \sqrt{2h}t v\|}{\sqrt{2h}t} \\
 & \le \lim_{t \to \infty} \lim_{n \to \infty} \left( \frac{\big\| \gm_n(t) -\|\gm_n(t)\|v  \big\|}{\sqrt{2h}t} + \frac{\big\| \|\gm_n(t)\|v - \sqrt{2h}tv \big\|}{\sqrt{2h}t} \right) \\
 & \le \lim_{t \to \infty} \lim_{n \to \infty} \left( \frac{\|\gm_n(t)\|}{\sqrt{2h}t} \lnm \frac{\gm_n(t)}{\|\gm_n(t)\|} -v \rnm + \left|\frac{\|\gm_n(t)\|}{\sqrt{2h}t} -1 \right|  \right) \\
 & = \lim_{t \to \infty} \left(\frac{\|\gm(t)\|}{\sqrt{2h}t} \lnm \frac{\gm(t)}{\|\gm(t)\|} -v \rnm + \left|\frac{\|\gm(t)\|}{\sqrt{2h}t} -1 \right| \right) \\
 & \le \lim_{t \to \infty} \left( 2 \cal \sum_{i = [\log \|\gm(t)\|]}^{\infty} \fal(2^{i}) + g_{\al}(\|\gm(t)\|) \right) =0. 
\end{aligned}
$$ 
This means $\gm(t) = \sqrt{2h}tv +o(t)$, when $t \to \infty$, which finishes our proof of Theorem \ref{thm;hyper}.

\noindent{\bf Acknowledgement.} The author thanks Duokui Yan and Yuan Zhou for their encouragement and helpful comments, and the anonymous referees for their value comments and suggestions.

\bibliographystyle{abbrv}
\bibliography{RefScattering}

\end{document}